\documentclass[11pt]{amsart}

\usepackage[T1]{fontenc}
\usepackage[utf8]{inputenc}
\usepackage{lmodern}
\usepackage{amsmath,amssymb,amsthm,mathtools,mathrsfs}
\usepackage[a4paper,margin=30mm]{geometry}
\usepackage[hidelinks]{hyperref}
\usepackage[nameinlink,capitalise,noabbrev]{cleveref}

\newtheorem{theorem}{Theorem}[section]
\newtheorem{proposition}[theorem]{Proposition}
\newtheorem{lemma}[theorem]{Lemma}
\newtheorem{conjecture}[theorem]{Conjecture}
\theoremstyle{definition}

\theoremstyle{remark}

\newcommand{\Sym}{\Lambda}

\newcommand{\cP}{\mathcal P}
\newcommand{\cQ}{\mathcal Q}
\newcommand{\cH}{\mathcal H}

\newcommand{\Hall}{\mathrm{Hall}}

\title{Jack Content Operators and the Deformed
\texorpdfstring{$\mathcal W_{1+\infty}$}{W(1+infinity)} Algebra}
\author{Jean-Yves Thibon}
\address{Laboratoire d'Informatique Gaspard-Monge,
Universit\'e Gustave Eiffel, CNRS, ESIEE Paris,
F-77454 Marne-la-Vall\'ee, France}
\email{jean-yves.thibon@univ-eiffel.fr}
\subjclass[2020]{Primary 05E05; Secondary 17B65, 20C30, 33D52}
\keywords{Jack polynomials, contents, cut-and-join operators,
Virasoro algebra, deformed $\mathcal W_{1+\infty}$ algebra,
spherical degenerate double affine Hecke algebra, affine Yangian}
\date{}

\begin{document}
\begin{abstract}
Frenkel and Wang obtained a representation of the Virasoro algebra by
commuting Goulden's cut-and-join operator with the Heisenberg generators.
A vertex-operator construction in \cite{LascouxThibon} extends this
representation to $\mathcal W_{1+\infty}$ by means of differential operators whose
eigenvalues are the power sums of the contents of a Young diagram.  We
develop a Jack deformation in the spherical degenerate double affine
Hecke algebra and its stable limit.  Starting from the
Heckman--Polychronakos integrals, we isolate
operators whose eigenvalues are the power sums of the $\alpha$-contents.
Because the Goulden--Jackson product is defined in the convention dual to
the usual Calogero--Sutherland Hamiltonians, the multiplication operators
$\Delta_\mu(\alpha)$ are obtained by taking Hall adjoints.  This gives
conceptual derivations of the Jack cut-and-join operator and of the
stable $3$-cycle operator.  A normal-ordering construction due to
Sergeev and Veselov makes the latter calculation explicit and suggests
an integral form over $\mathbb Z[\alpha]$.  The commutators of the
cut-and-join operator contain one half of the usual Feigin--Fuchs
realization, but this Virasoro completion is not the deformation of the
Frenkel--Wang construction.  The latter takes place in the deformed
$\mathcal W_{1+\infty}$ algebra $\mathbf{SH}^c$, equivalently in the affine
Yangian of $\mathfrak{gl}_1$: in our normalization its first nontrivial
Cartan mode is
$\psi_3=3\Delta_2(\alpha)+2(\alpha-1)E$, and its commutators with the
first raising and lowering modes recursively generate the remaining
currents.  At $\alpha=1$ these relations
specialize to the central-charge-one $\mathcal W_{1+\infty}$
representation used by Lascoux and the author.
\end{abstract}
\maketitle

\section{Introduction}

The direct sum of the class algebras of the symmetric groups admits a
natural Fock-space realization.  Under the Frobenius characteristic map,
convolution by the sum of all transpositions becomes Goulden's
cut-and-join operator.  Frenkel and Wang observed that its commutators
with the creation and annihilation operators of the Heisenberg algebra
produce the modes of a Virasoro algebra of central charge one
\cite{FrenkelWang}.

The transposition sum can also be written as the first power sum of the
Jucys--Murphy elements.  This observation suggests replacing it by all
power sums $p_r(\Xi_n)$.  The construction of \cite{LascouxThibon}
uses a single vertex operator encoding the corresponding
differential operators for all $r$ and all symmetric groups.  Its
coefficients generate a representation of $\mathcal W_{1+\infty}$.
Since symmetric functions of the Jucys--Murphy elements generate the
center, the same formalism also contains the operators of multiplication
by normalized conjugacy classes and explains their stability.

The aim of the present paper is to deform this picture from Schur
functions to Jack functions.  The finite-rank home of the construction
is not an arbitrarily deformed class algebra.  The commuting
Dunkl--Cherednik operators form the spectral commutative subalgebra of
a degenerate affine Hecke algebra, and their symmetric polynomials act
through the spherical degenerate double affine Hecke algebra.  Its
stable limit is the algebra $\mathbf{SH}^c$ of Schiffmann and Vasserot,
a deformation of $U(\mathcal W_{1+\infty})$
\cite{SchiffmannVasserotCherednik,ArbesfeldSchiffmann}.  After the usual
parameter identification and completion, this is also the affine
Yangian of $\mathfrak{gl}_1$ \cite{Tsymbaliuk}.  Thus the undeformed
$U(\mathcal W_{1+\infty})$ statement belongs to the Schur
specialization; the generic Jack construction lives naturally in its
Yangian deformation.

There are two equivalent, but conventionally opposite, sides of the
problem.  On the spectral side, the
Calogero--Sutherland, Sekiguchi--Debiard and Dunkl--Heckman Hamiltonians
act diagonally on ordinary Jack functions \cite{NazarovSklyanin,SergeevVeselov}.
On the combinatorial side, the Jack deformation of the class-algebra
product is naturally expressed in a basis dual to the ordinary Jack
basis.  Consequently, the operators $\Delta_\mu(\alpha)$ used here are
the \emph{Hall adjoints} of the corresponding spectral Hamiltonians.
This reversal is harmless, but it must be kept visible throughout the
paper.

Our starting point is the family of Heckman--Polychronakos operators
\begin{equation}\label{eq:intro-HP}
 \cP_m^{(N)}=\sum_{i=1}^N(x_i\mathcal D_i)^m,
 \qquad
 \mathcal D_i=\frac{\partial}{\partial x_i}
 +\theta\sum_{j\ne i}\frac{1-s_{ij}}{x_i-x_j},
 \qquad \theta=\alpha^{-1}.
\end{equation}
Their explicit spectrum on Jack polynomials was recently written in a
particularly useful generating-function form \cite{DunklGorin}.  After
a translation removing the number of
variables, division by the vacuum eigenvalue, and a logarithm, this
generating series yields commuting operators $\cQ_r$ whose eigenvalues
are finite-difference power sums of the shifted row coordinates.  The
ordinary power sums of $\alpha$-contents are then finite triangular
linear combinations of the $\cQ_r$, obtained from Bernoulli polynomials.

The first identity is
\begin{equation}\label{eq:intro-Delta2}
 \Delta_2(\alpha)=
 \left(
  \frac{\alpha^2}{3}\cQ_3
  +\frac{\alpha(1-\alpha)}2\cQ_2
 \right)^{\perp_{\Hall}}.
\end{equation}
It recovers the Goulden--Jackson deformed cut-and-join operator.  Its commutators
with the Heisenberg modes contain the expected quadratic expressions
together with the linear correction characteristic of the Jack
parameter.  One half of these operators has the Feigin--Fuchs form and
therefore admits the familiar Virasoro completion.  This observation is
useful, but it is not by itself the Jack analogue of the
Frenkel--Wang construction: that analogue requires the raising, lowering
and Cartan currents of the affine Yangian.  The defining Yangian
relations give a recursive completion of the commutator calculation,
and we determine explicitly the normalization of the first Cartan,
raising and lowering modes.

At the next order, the normally ordered Dunkl--Heckman operator at
infinity of Sergeev and Veselov gives a closed expression for
$\cQ_4$.  This proves an explicit formula for $\Delta_3(\alpha)$ and
reveals a cancellation mechanism for the auxiliary dimension and the
negative powers of $\alpha$.  In the first two nontrivial cases, the
resulting coefficients belong to $\mathbb Z[\alpha]$.

The paper is organized as follows.  We first recall the classical
cut-and-join, Virasoro and $\mathcal W_{1+\infty}$ constructions.  We
then fix the Jack and duality conventions and identify the spherical
Cherednik framework.  The stable logarithmic
Hamiltonians $\cQ_r$ are constructed from Heckman--Polychronakos
operators, and their spectra are converted into content moments.  The
Sergeev--Veselov recursion is then used to normally order $\cQ_4$.  We
derive $\Delta_2(\alpha)$ and $\Delta_3(\alpha)$, discuss the resulting
integral normal forms, and compare the commutators of
$\Delta_2(\alpha)$ with the Feigin--Fuchs modes.  We finally identify the
correct completion mechanism in the Fock representation of the affine
Yangian, and explain how it specializes to the classical construction of
Frenkel--Wang, Lascoux and the author..

\section{The classical construction}\label{sec:classical}

\subsection{Heisenberg modes and Hall duality}

Let $\Sym=\mathbb Q[p_1,p_2,\ldots]$, endowed with the Hall scalar
product
\begin{equation}
 \langle p_\lambda,p_\mu\rangle_{\Hall}
 =\delta_{\lambda\mu}z_\lambda.
\end{equation}
We use the Heisenberg modes
\begin{equation}
 a_{-k}=p_k,\qquad
 a_k=D_k:=k\frac{\partial}{\partial p_k},
 \qquad k\geq1,
\end{equation}
so that $a_k=a_{-k}^{\perp_{\Hall}}$ and
$[a_k,a_{-\ell}]=k\delta_{k\ell}$.

\subsection{Goulden's operator and Virasoro}

Let $K_{(2,1^{n-2})}$ be the characteristic function of the
transposition class of $\mathfrak S_n$.  Under the Frobenius
characteristic map, convolution by this function is represented by
Goulden's operator \cite{Goulden}
\begin{equation}\label{eq:classical-cutjoin}
 \mathscr G=\frac12
 \sum_{i,j\geq1}
 \left(p_ip_jD_{i+j}+p_{i+j}D_iD_j\right).
\end{equation}
Indeed, multiplication of a permutation by a transposition either joins
two cycles or cuts one cycle into two.  The first summand in
\eqref{eq:classical-cutjoin} records joins and the second one cuts.  The
factor $1/2$ removes the double counting of the ordered pair $(i,j)$.

Our stable series are normalized by reduced cycle type.  With the
notation used in the rest of this paper, this amounts to
\begin{equation}\label{eq:Delta2-vs-Goulden}
 \Delta_2(1)=2\mathscr G.
\end{equation}
This factor two will be retained throughout; it is the first of the two
normalization differences with the integrable-systems literature, the
second being Hall duality.

Extend the Heisenberg modes to all nonzero integers by the conventions
above and define
\begin{equation}\label{eq:classical-Virasoro}
 L_k=\frac12\sum_{r\in\mathbb Z}:a_{k-r}a_r:,
 \qquad a_0=0.
\end{equation}
Normal ordering places the creation operators to the left of the
annihilation operators.  A direct calculation gives
\begin{equation}\label{eq:Goulden-Heisenberg}
 [\mathscr G,a_k]=-kL_k,
 \qquad k\ne0.
\end{equation}
Moreover,
\begin{align}
 [L_k,a_l]&=-l a_{k+l},\label{eq:classical-La}\\
 [L_k,L_l]&=(k-l)L_{k+l}
 +\frac{k^3-k}{12}\delta_{k,-l}.\label{eq:classical-Vir-rel}
\end{align}
Thus the cut-and-join operator, together with the Heisenberg modes,
generates a Virasoro algebra of central charge one.  Frenkel and Wang
obtained this construction group-theoretically, more generally for
wreath products \cite{FrenkelWang}.

\subsection{Jucys--Murphy elements and content eigenvalues}

For $2\leq i\leq n$, let
\begin{equation}
 \xi_i=(1,i)+(2,i)+\cdots+(i-1,i),
 \qquad \xi_1=0,
\end{equation}
and write $\Xi_n=\{\xi_1,\ldots,\xi_n\}$.  Symmetric polynomials in the
$\xi_i$ are central \cite{Jucys}.  On the irreducible module indexed by $\lambda$,
the central element $f(\Xi_n)$ acts by the scalar
\begin{equation}\label{eq:JM-content-spectrum}
 f(C(\lambda)),
 \qquad
 C(\lambda)=\{j-i:(i,j)\in\lambda\}.
\end{equation}
In particular, the sum of all transpositions is
$p_1(\Xi_n)$ and its eigenvalue is the sum of the contents.

Let $\mathscr D_r$ denote the operator, acting simultaneously on all
homogeneous components of $\Sym$, which represents multiplication by
$p_r(\Xi_n)$.  It is convenient to collect these operators into
\begin{equation}\label{eq:classical-D-generating}
 \mathscr D(t)=\sum_{r\geq1}\mathscr D_r\frac{t^r}{r!}.
\end{equation}
Its eigenvalue on $s_\lambda$ is
\begin{equation}\label{eq:classical-D-spectrum}
 \sum_{\square\in\lambda}
 \bigl(e^{t c(\square)}-1\bigr).
\end{equation}
Notice that $\mathscr D_1=\mathscr G$ with the Frenkel--Wang
normalization.

\subsection{The vertex operator construction}

Put $q=e^t$ and introduce the normally ordered vertex operator
\begin{equation}\label{eq:LT-vertex}
\begin{split}
 V(z;q)
 &={}:\exp\left(
   \sum_{k\ne0}\frac{1-q^{-k}}{k}z^{-k}a_k
  \right):\\
 &=\exp\left(
   \sum_{k\geq1}(q^k-1)p_k\frac{z^k}{k}
  \right)
  \exp\left(
   \sum_{k\geq1}(1-q^{-k})z^{-k}
   \frac{\partial}{\partial p_k}
  \right).
\end{split}
\end{equation}
If $V_0(q)$ is its constant term in $z$, then
\begin{equation}\label{eq:LT-D-vertex}
 \mathscr D(t)
 =\frac{V_0(q)-1}{(q-1)(1-q^{-1})}-E,
 \qquad
 E=\sum_{k\geq1}p_kD_k.
\end{equation}
Equivalently, using plethystic notation and the complete functions,
\begin{equation}\label{eq:LT-D-differential}
 \mathscr D(t)=
 \frac{q}{(q-1)^2}
 \sum_{m\geq1}q^{-m}
 h_m[(q-1)X]D_{h_m[(q-1)X]}-E.
\end{equation}
Formula \eqref{eq:LT-D-vertex} is the bosonization of the content
eigenvalue \eqref{eq:classical-D-spectrum}.

The elementary commutator
\begin{equation}\label{eq:V-Heisenberg}
 [V(z;q),a_k]=z^k(1-q^k)V(z;q)
\end{equation}
shows that all modes of $V(z;q)$ are generated from its zero mode by
commutation with Heisenberg.  Define $T_k(q)$ by the normalization
\begin{equation}
 T_k(q)=-\frac{q^{-1}}{1-q^{-1}}[z^{-k}](V(z;q)-1).
\end{equation}
Their commutation relations are
\begin{equation}\label{eq:Tk-commutation}
\begin{split}
 [T_k(a),T_l(b)]={}&(a^l-b^k)T_{k+l}(ab)\\
 &+\delta_{k,-l}\frac{a^{-k}-b^{-l}}{1-ab}.
\end{split}
\end{equation}
These are the relations of the central extension of the Lie algebra of
differential operators on the circle.  Expanding
\begin{equation}
 T_k(e^t)=\sum_{r\geq0}\frac{t^r}{r!}T_{k,r}
\end{equation}
identifies the $T_{k,r}$ with the standard generators of a charge-one
representation of $\mathcal W_{1+\infty}$.  In particular,
\begin{equation}
 T_{k,1}=a_k,
 \qquad T_{k,2}=2L_k.
\end{equation}
Since each $\mathscr D_r$ is a triangular linear combination of the
zero modes $T_{0,s}$, the commutators
\begin{equation}
 [\mathscr D_r,a_k],
 \qquad r\geq1,\quad k\ne0,
\end{equation}
generate $\mathcal W_{1+\infty}$ \cite{LascouxThibon}.  For $r=1$ this
reduces to \eqref{eq:Goulden-Heisenberg}.

\subsection{Stable conjugacy classes}

We recall the symmetric-function formulation of class multiplication,
following \cite{LascouxThibon,ThibonStable}.  For a fixed $n$, the
Frobenius characteristic identifies the centre $Z\mathbb C\mathfrak S_n$
with $\Lambda_n$.  We denote the transported product by $\times$ and call
it the \emph{class product}.  Thus
\begin{equation}\label{eq:class-product-Schur}
 s_\lambda\times s_\mu
 =\delta_{\lambda\mu}\frac{s_\lambda}{f^\lambda},
 \qquad \lambda,\mu\vdash n,
\end{equation}
where $f^\lambda$ is the dimension of the irreducible representation of
shape $\lambda$.

There is a useful coalgebraic reformulation.  Put
\begin{equation}\label{eq:class-kernel}
 \Phi_n(X,Y,Z)=
 \sum_{\lambda\vdash n}\frac{1}{f^\lambda}
 s_\lambda(X)s_\lambda(Y)s_\lambda(Z)
\end{equation}
and contract the $Z$ alphabet with the Hall scalar product:
\begin{equation}\label{eq:class-coproduct-kernel}
 \Gamma_n(f)(X,Y)
 =\bigl\langle\Phi_n(X,Y,Z),f(Z)\bigr\rangle_{\Hall,Z}.
\end{equation}
Since the Schur functions are Hall orthonormal,
\begin{equation}\label{eq:class-coproduct-Schur}
 \Gamma_n(s_\lambda)=\frac1{f^\lambda}s_\lambda\otimes s_\lambda.
\end{equation}
The product and coproduct are Hall dual:
\begin{equation}\label{eq:class-duality}
 \langle f\times g,h\rangle_{\Hall}
 =\langle f\otimes g,\Gamma_n(h)\rangle_{\Hall\otimes\Hall}.
\end{equation}
Equations \eqref{eq:class-product-Schur} and
\eqref{eq:class-coproduct-Schur} show immediately that the two
definitions agree.

To place all degrees in one object, let
\begin{equation}
 \widehat\Lambda=\prod_{n\geq0}\Lambda_n
\end{equation}
and extend $\times$ componentwise.  If $f\in\Lambda_m$ is homogeneous,
its \emph{stable lift} is
\begin{equation}\label{eq:stable-lift-definition}
 \widehat f=\frac{m!f}{(1-p_1)^{m+1}}
 =\sum_{n\geq m}(n)_m f p_1^{n-m}.
\end{equation}
In particular,
\begin{equation}\label{eq:stable-power-sum-definition}
 \widehat p_\mu
 =\sum_{n\geq|\mu|}(n)_{|\mu|}
 p_{\mu,1^{n-|\mu|}}.
\end{equation}
Parts equal to $1$ in $\mu$ are retained: combinatorially, they mark
distinguished fixed points.  When $\mu$ has no part equal to $1$, the
degree-$n$ component is the Frobenius characteristic of the corresponding
normalized conjugacy class.  The span of the $\widehat p_\mu$ is closed
under $\times$ and realizes the Ivanov--Kerov stable class algebra.  For
example,
\begin{equation}\label{eq:stable-p2-classical-example}
 \widehat p_2\times\widehat p_2
 =\widehat p_{22}+4\widehat p_3+2\widehat p_{11}.
\end{equation}

The preceding construction contains more than the content power sums.
By a theorem of Kerov and Olshanski \cite{KerovOlshanski}, normalized
conjugacy classes are polynomial functions of the shifted
row coordinates
\begin{equation}\label{eq:shifted-row-powers}
 \widetilde p_r(\lambda)=
 \sum_{i\geq1}\bigl((\lambda_i-i)^r-(-i)^r\bigr).
\end{equation}
The operators with eigenvalues \eqref{eq:shifted-row-powers} are
triangular linear combinations of the $\mathscr D_r$.  It follows that
the multiplication operators associated with normalized conjugacy
classes belong to the image of
$U(\mathcal W_{1+\infty})$.  This is the classical statement that we
shall deform: the Jack content observables must first be reconstructed
from a commuting spectral hierarchy, after which Hall duality turns
them into the multiplication operators $\Delta_\mu(\alpha)$.

\section{Jack conventions and the Hall-adjoint reversal}
\label{sec:conventions}

We now recall the Goulden--Jackson deformation
\cite{GouldenJacksonJack,ThibonStable}.  The Jack scalar product is
\begin{equation}\label{eq:Jack-scalar-product}
 \langle p_\lambda,p_\mu\rangle_\alpha
 =\delta_{\lambda\mu}z_\lambda\alpha^{\ell(\lambda)}.
\end{equation}
Let $J_\lambda^{(\alpha)}$ be the integral Jack polynomial and write
$j_\lambda(\alpha)=\langle J_\lambda^{(\alpha)},
J_\lambda^{(\alpha)}\rangle_\alpha$.  For $\lambda\vdash n$, set
\begin{equation}\label{eq:GJ-kernel}
 \Phi_{\alpha,n}(X,Y,Z)=
 \sum_{\lambda\vdash n}
 \frac{J_\lambda^{(\alpha)}(X)J_\lambda^{(\alpha)}(Y)
 J_\lambda^{(\alpha)}(Z)}{j_\lambda(\alpha)}.
\end{equation}
The Jack connection coefficients $a_{\mu\nu}^{\rho}(\alpha)$ are defined
by
\begin{equation}\label{eq:GJ-connection-coefficients}
 \Phi_{\alpha,n}(X,Y,Z)
 =\sum_{\rho,\mu,\nu\vdash n}a_{\mu\nu}^{\rho}(\alpha)
 \frac{p_\rho(X)p_\mu(Y)p_\nu(Z)}
 {z_\rho\alpha^{\ell(\rho)}}.
\end{equation}

Equivalently, define a coproduct on $\Lambda_n$ by
\begin{equation}\label{eq:GJ-coproduct}
 \Gamma_\alpha(J_\lambda^{(\alpha)})
 =\frac1{n!}J_\lambda^{(\alpha)}\otimes J_\lambda^{(\alpha)}
\end{equation}
and define $\times_\alpha$ as its dual for the \emph{ordinary Hall}
scalar product:
\begin{equation}\label{eq:GJ-Hall-duality}
 \langle f\times_\alpha g,h\rangle_{\Hall}
 =\langle f\otimes g,\Gamma_\alpha(h)\rangle_{\Hall\otimes\Hall}.
\end{equation}
This choice deserves explanation.  Jack orthogonality is used in
\eqref{eq:GJ-kernel}, whereas Hall duality is used in
\eqref{eq:GJ-Hall-duality}.  It preserves the same power-sum
coordinates as in the class algebra and makes the specialization
$\alpha=1$ literally equal to $\times$.  Had we dualized with
\eqref{eq:Jack-scalar-product}, factors $\alpha^{\ell(\mu)}$ would be
inserted into the power-sum basis.

There is also a concrete reason for this convention.  At
$\alpha=1,2,1/2$, the same formal trace polynomial is evaluated by a
Gaussian integral over, respectively, complex, real, and quaternionic
matrix spaces.  Wick expansion produces the coefficients of the
coproduct in the unchanged power-sum coordinates; only the matrix space
and its covariance change.  Thus Hall duality keeps the three matrix
models in a common notation.  Detailed normalization checks for these
integrals are given in \cite{ThibonStable}.

The coproduct is cocommutative and coassociative, hence
$\times_\alpha$ is a commutative associative product.  If
$Q'_\lambda$ denotes the basis Hall-dual to the monic Jack basis
$P_\lambda^{(\alpha)}$, then
\begin{equation}\label{eq:GJ-idempotents}
 Q'_\lambda\times_\alpha Q'_\mu
 =\delta_{\lambda\mu}\frac{c_\lambda(\alpha)}{n!}Q'_\lambda.
\end{equation}
This gives, in particular, a direct proof of
associativity.

For the Jack deformation we put
\begin{equation}\label{eq:jack-heisenberg}
 a_{-k}=\alpha p_k,
 \qquad a_k=D_k,
 \qquad [a_k,a_{-\ell}]=\alpha k\delta_{k\ell}.
\end{equation}
Our operator $\Delta_\mu(\alpha)$ is, by definition, multiplication by
the stable series $\widehat p_\mu$ for the Goulden--Jackson product.
This convention is dual to the usual spectral convention.  We therefore
reserve $\cH_\mu(\alpha)$ for the Hamiltonian acting diagonally on
ordinary Jack functions and impose
\begin{equation}\label{eq:adjoint-convention}
  {\Delta_\mu(\alpha)
 =\cH_\mu(\alpha)^{\perp_{\Hall}}.}
\end{equation}

The remaining sections use this convention.

As a first normalization check, the explicit Goulden--Jackson operator
\cite{GouldenJacksonJack} gives
\begin{equation}\label{eq:stable-p2-Jack-example}
 \widehat p_2\times_\alpha\widehat p_2
 =\widehat p_{22}+4\widehat p_3
 +2\alpha\widehat p_{11}+2(\alpha-1)\widehat p_2.
\end{equation}
At $\alpha=1$ this reduces to
\eqref{eq:stable-p2-classical-example}; the last term is the first
contribution with no permutation-only interpretation.

\section{The spherical Cherednik framework}
\label{sec:ambient-algebra}

Before introducing Cherednik algebras, let us identify precisely the
algebra whose connection coefficients occur above.  For each $n$, put
\begin{equation}\label{eq:GJ-finite-algebra}
 \mathcal A_n^{(\alpha)}=(\Lambda_n,\times_\alpha),
 \qquad K_\mu=\frac{n!}{z_\mu}p_\mu\quad(\mu\vdash n).
\end{equation}
Then \eqref{eq:GJ-connection-coefficients} and
\eqref{eq:GJ-Hall-duality} give
\begin{equation}\label{eq:GJ-structure-algebra}
 K_\mu\times_\alpha K_\nu
 =\sum_{\rho\vdash n}a_{\mu\nu}^{\rho}(\alpha)K_\rho.
\end{equation}
Indeed, contracting \eqref{eq:GJ-kernel} with the Jack scalar product
in its third alphabet gives $n!\Gamma_\alpha$.  The symmetry of the
coefficient of $p_\rho(X)p_\mu(Y)p_\nu(Z)$, followed by Hall duality,
gives \eqref{eq:GJ-structure-algebra}.  Thus the
$a_{\mu\nu}^{\rho}(\alpha)$ are, literally, the structure constants of
the finite-dimensional commutative semisimple algebra
$\mathcal A_n^{(\alpha)}$ in the basis $K_\mu$.

For $\alpha=1$, this is the centre of $\mathbb C\mathfrak S_n$ under
Frobenius characteristic.  For $\alpha=2$, it is the zonal or matching
analogue associated with the Gelfand pair
$(\mathfrak S_{2n},H_n)$, where $H_n$ is the hyperoctahedral group.
For $\alpha=1/2$, the quaternionic zonal matrix model supplies the dual
classical specialization.  For a generic parameter,
$\mathcal A_n^{(\alpha)}$ is best regarded as the abstract
Goulden--Jackson algebra defined by
\eqref{eq:GJ-coproduct}--\eqref{eq:GJ-Hall-duality}.

%This algebra must not be confused with the spherical Cherednik algebra.
%The former is commutative and records the deformed class product; the
%latter is a larger noncommutative operator algebra in which its regular
%multiplication operators, together with Heisenberg creation and
%annihilation operators, can be realized.

We now specify the algebra in which the deformation takes place.  Let
$\mathsf H_N(\theta)$ denote the type-$A$ degenerate double affine
Hecke algebra in its polynomial representation, and let
\begin{equation}
 e_N=\frac1{N!}\sum_{w\in\mathfrak S_N}w
\end{equation}
be the symmetrizing idempotent.  The algebra acting on
$\Lambda_N$ is the spherical algebra
\begin{equation}\label{eq:spherical-ddaha}
 e_N\mathsf H_N(\theta)e_N.
\end{equation}
The commuting Cherednik elements generate a degenerate affine Hecke
subalgebra.  Symmetric polynomials in these elements are central in
that subalgebra; after spherical projection they give the
Sekiguchi--Debiard and Dunkl--Heckman Hamiltonians.  Thus it is useful
to think of the content operators as coming from the centre of the
degenerate affine Hecke algebra, but the full algebra containing both
the Hamiltonians and the creation and annihilation operators is
\eqref{eq:spherical-ddaha}, not that centre alone.

Schiffmann and Vasserot construct a stable, centrally extended limit of
the algebras \eqref{eq:spherical-ddaha}, denoted $\mathbf{SH}^c$
\cite{SchiffmannVasserotCherednik}.  Arbesfeld and Schiffmann give it a
presentation as a deformed $W_{1+\infty}$ algebra
\cite{ArbesfeldSchiffmann}.  The polynomial representations stabilize
to its rank-one Fock representation
\begin{equation}\label{eq:SHc-Fock}
 \mathcal F=\mathbb Q(\alpha)[p_1,p_2,\ldots],
\end{equation}
whose Heisenberg action is \eqref{eq:jack-heisenberg}.  The stable
Hamiltonians constructed below belong to the completed degree-zero
commutative subalgebra of this representation.

There is an equivalent Yangian language.  Up to an overall rescaling
and a permutation of parameters, our convention corresponds to
\begin{equation}\label{eq:Yangian-parameters}
 (h_1,h_2,h_3)=(\alpha,-1,1-\alpha),
 \qquad h_1+h_2+h_3=0,
\end{equation}
because the weight of a box $(i,j)$ is
$h_1(j-1)+h_2(i-1)=c_\alpha(i,j)$.  The algebra $\mathbf{SH}^c$ is
related, after the standard completion and parameter identification,
to the affine Yangian of $\mathfrak{gl}_1$; see
\cite{Tsymbaliuk}.  In this description the commuting content
Hamiltonians are Cartan modes, while multiplication by $p_1$ and its
adjoint give the first raising and lowering modes.

At $\alpha=1$ one has $h_3=0$.  The deformation then specializes to
the enveloping algebra of the central extension of differential
operators on the circle, namely $U(\mathcal W_{1+\infty})$.  Under
boson--fermion correspondence, \eqref{eq:SHc-Fock} is the charge-zero
sector of the level-one fermionic Fock representation of
$\widehat{\mathfrak{gl}}_\infty$, and $\mathcal W_{1+\infty}$ is the
differential-operator subalgebra used in \cite{LascouxThibon}.  For
generic $\alpha$, it is therefore more accurate to speak of the
$\mathbf{SH}^c$ or affine-Yangian Fock representation than of a new
representation of the undeformed $U(\widehat{\mathfrak{gl}}_\infty)$.

\section{Stable Dunkl--Heckman Hamiltonians}
\label{sec:stable-Dunkl}

We now construct the commuting operators that will replace the diagonal
content operators of the Schur case.  We first work in a fixed number of
variables.  The dependence on this number is then removed by a
translation and a vacuum normalization.

\subsection{Heckman--Polychronakos operators}

Let
\begin{equation}
 \Lambda_N=\mathbb Q(\theta)[x_1,\ldots,x_N]^{\mathfrak S_N},
 \qquad \theta=\alpha^{-1}.
\end{equation}
On the full polynomial ring, define the Dunkl operators
\begin{equation}\label{eq:Dunkl-finite}
 \mathcal D_i=\frac{\partial}{\partial x_i}
 +\theta\sum_{j\ne i}\frac{1-s_{ij}}{x_i-x_j},
 \qquad 1\leq i\leq N,
\end{equation}
where $s_{ij}$ exchanges $x_i$ and $x_j$.  The
Heckman--Polychronakos operators are
\begin{equation}\label{eq:HP-finite}
 \cP_m^{(N)}=\sum_{i=1}^N(x_i\mathcal D_i)^m,
 \qquad m\geq1,
\end{equation}
and we put $\cP_0^{(N)}=N\,\mathrm{Id}$.  Although the individual
operators $x_i\mathcal D_i$ do not commute, the sums
$\cP_m^{(N)}$ commute with one another and preserve $\Lambda_N$
\cite{SergeevVeselov,DunklGorin}.

Let $P_\lambda^{(\alpha)}(x_1,\ldots,x_N)$ denote a Jack polynomial,
where $\ell(\lambda)\leq N$.  Write
\begin{equation}\label{eq:ell-i}
 \ell_i=\lambda_i+\theta(N-i),
 \qquad 1\leq i\leq N,
\end{equation}
and let $\operatorname{eig}_m^{(N)}(\lambda)$ be the eigenvalue of
$\cP_m^{(N)}$ on $P_\lambda^{(\alpha)}$.  We shall use the following
form of the spectral formula of Dunkl and Gorin \cite{DunklGorin}:
\begin{equation}\label{eq:DG-generating-spectrum}
 1-\theta z\sum_{m\geq0}
 \operatorname{eig}_m^{(N)}(\lambda)z^m
 =\prod_{i=1}^N
 \frac{1-(\ell_i+\theta)z}{1-\ell_i z}.
\end{equation}
The convention $\operatorname{eig}_0^{(N)}=N$ makes the constant and
linear terms agree.

\subsection{Translation and vacuum normalization}

The coordinates \eqref{eq:ell-i} depend on $N$.  Define translated
operators
\begin{equation}\label{eq:Pbar-definition}
 \overline{\cP}_m^{(N)}
 =\sum_{r=0}^m\binom mr
 (-\theta N)^{m-r}\cP_r^{(N)}.
\end{equation}
This is the operator counterpart of translating every $\ell_i$ by
$-\theta N$.  Indeed, the matrix formula for the eigenvalues in
\cite{DunklGorin} shows that
\begin{equation}
 \overline{\operatorname{eig}}{}_m^{(N)}(\lambda)
 =\operatorname{eig}_m^{(N)}
   (\ell_1-\theta N,\ldots,\ell_N-\theta N).
\end{equation}

Consider the operator-valued series
\begin{equation}\label{eq:Gbar-definition}
 \overline G_N(z)
 =1-\theta z\sum_{m\geq0}
   \overline{\cP}_m^{(N)}z^m.
\end{equation}
Put
\begin{equation}\label{eq:A-coordinates}
 A_i(\lambda)=\lambda_i-\theta(i-1),
 \qquad A_i(0)=-\theta(i-1).
\end{equation}
Since
\begin{equation}
 \ell_i-\theta N=A_i(\lambda)-\theta,
 \qquad
 \ell_i+\theta-\theta N=A_i(\lambda),
\end{equation}
formula \eqref{eq:DG-generating-spectrum} gives
\begin{equation}\label{eq:Gbar-spectrum}
 \overline G_N(z)P_\lambda^{(\alpha)}
 =\prod_{i=1}^N
  \frac{1-A_i(\lambda)z}
       {1-(A_i(\lambda)-\theta)z}
  P_\lambda^{(\alpha)}.
\end{equation}
For the empty partition, this product telescopes:
\begin{equation}\label{eq:vacuum-telescope}
 \overline G_N(z)1
 =\prod_{i=1}^N
  \frac{1+\theta(i-1)z}{1+\theta i z}
 =\frac1{1+\theta Nz}.
\end{equation}
We therefore define the vacuum-normalized series
\begin{equation}\label{eq:PhiN-definition}
 \Phi_N(z)
 =\frac{\overline G_N(z)}{\overline G_N(z)1}
 =(1+\theta Nz)\overline G_N(z).
\end{equation}

\begin{proposition}[Stability]\label{prop:Phi-stability}
For every partition $\lambda$ with $\ell(\lambda)\leq N$, the
eigenvalue of $\Phi_N(z)$ on $P_\lambda^{(\alpha)}$ is
\begin{equation}\label{eq:Phi-stable-spectrum}
 \Phi_\lambda(z)=
 \prod_{i\geq1}
 \frac{1-A_i(\lambda)z}{1-(A_i(\lambda)-\theta)z}
 \frac{1-(A_i(0)-\theta)z}{1-A_i(0)z}.
\end{equation}
The product is finite and independent of $N$.  Consequently, the
operators $\Phi_N(z)$ are compatible with the specialization
$x_N=0$ and define an operator $\Phi(z)$ on $\Sym$.
\end{proposition}

\begin{proof}
Divide \eqref{eq:Gbar-spectrum} by its value at the empty partition.
This gives the product in \eqref{eq:Phi-stable-spectrum} truncated at
$i=N$.  If $i>\ell(\lambda)$, then $A_i(\lambda)=A_i(0)$, so that the
corresponding factor is one.  The product is therefore unchanged when
$N$ is increased.

The Jack stability relation
\begin{equation}
 P_\lambda^{(\alpha)}(x_1,\ldots,x_{N-1},0)
 =P_\lambda^{(\alpha)}(x_1,\ldots,x_{N-1})
\end{equation}
then proves the compatibility of the operators, since the Jack
polynomials form a basis.
\end{proof}

\subsection{Logarithmic Hamiltonians}

The coefficients of $\Phi(z)$ commute and its constant term is the
identity.  We may thus define commuting stable operators $\cQ_r$ by
\begin{equation}\label{eq:Q-log-definition}
 -\log\Phi(z)=\sum_{r\geq1}\frac{\cQ_r}{r}z^r.
\end{equation}
The vacuum normalization implies $\cQ_r1=0$.  Taking the logarithm of
\eqref{eq:Phi-stable-spectrum} gives the following explicit spectrum.

\begin{theorem}\label{thm:Q-spectrum}
The Jack polynomials are joint eigenfunctions of the operators
$\cQ_r$.  Their eigenvalues are
\begin{equation}\label{eq:q-r-spectrum}
 q_r(\lambda)=\sum_{i\geq1}
 \left[
 A_i(\lambda)^r-(A_i(\lambda)-\theta)^r
 -A_i(0)^r+(A_i(0)-\theta)^r
 \right].
\end{equation}
In particular, $q_1=0$, and the right-hand side is a shifted symmetric
function of $\lambda$.
\end{theorem}

\begin{proof}
Use
\begin{equation}
 -\log(1-uz)=\sum_{r\geq1}\frac{u^r}{r}z^r
\end{equation}
in each of the four factors in
\eqref{eq:Phi-stable-spectrum}.  Only the indices
$i\leq\ell(\lambda)$ contribute, so all coefficient extractions are
finite.
\end{proof}

For later use, introduce the shifted row power sums
\begin{equation}\label{eq:S-r-definition}
 S_r(\lambda)=\sum_{i\geq1}
 \bigl(A_i(\lambda)^r-A_i(0)^r\bigr).
\end{equation}
Expanding the finite difference in \eqref{eq:q-r-spectrum}, we obtain
\begin{equation}\label{eq:q-vs-S-general}
 q_r=\sum_{j=1}^{r-1}(-1)^{j+1}
 \binom rj\theta^jS_{r-j}.
\end{equation}
The first cases are
\begin{align}
 q_2&=2\theta S_1,\label{eq:q2-S}\\
 q_3&=3\theta S_2-3\theta^2S_1,\label{eq:q3-S}\\
 q_4&=4\theta S_3-6\theta^2S_2
       +4\theta^3S_1.\label{eq:q4-S}
\end{align}
Thus the $\cQ_r$ generate the same shifted-symmetric spectral algebra
as the $S_r$.

The first nonzero operators can already be written intrinsically on
$\Sym$, without introducing a finite number of variables.  With
$E=\sum_{k\geq1}p_kD_k$, one has
\begin{equation}\label{eq:Q2-stable-explicit}
  {\cQ_2=\frac{2}{\alpha}E}
\end{equation}
and
\begin{equation}\label{eq:Q3-stable-explicit}
  {
 \cQ_3=\frac{3}{\alpha^2}\left(
  \sum_{i,j\geq1}p_ip_jD_{i+j}
  +\alpha\sum_{i,j\geq1}p_{i+j}D_iD_j
  +(\alpha-1)\sum_{k\geq1}k p_kD_k
 \right).}
\end{equation}
Both sums are locally finite on $\Sym$.  Formula
\eqref{eq:Q2-stable-explicit} follows from $q_2=2\theta|\lambda|$.
Formula \eqref{eq:Q3-stable-explicit} follows by expanding the
vacuum-normalized generating series \eqref{eq:PhiN-definition};
equivalently, one substitutes the collective-variable expression for
$\cP_2^{(N)}$ into \eqref{eq:Q3-P}.  The terms depending on $N$ cancel.
These formulas also provide direct normalization checks for the
logarithmic definition \eqref{eq:Q-log-definition}.

\subsection{An exponential generating series}

The logarithmic Hamiltonians admit the exponential generating series
\begin{equation}\label{eq:Q-exponential-series}
 \cQ^{\exp}(u)=
 \sum_{r\geq1}\cQ_r\frac{u^r}{r!}.
\end{equation}
By \eqref{eq:q-r-spectrum}, its eigenvalue is
\begin{equation}\label{eq:Qexp-spectrum}
 Q_\lambda^{\exp}(u)
 =(1-e^{-\theta u})
 \sum_{i\geq1}
 \left(e^{uA_i(\lambda)}-e^{uA_i(0)}\right).
\end{equation}

For a box $\square=(i,j)$, put
\begin{equation}\label{eq:alpha-content}
 c_\alpha(\square)=\alpha(j-1)-(i-1),
\end{equation}
and define
\begin{equation}\label{eq:content-moment}
 C_r^{(\alpha)}(\lambda)
 =\sum_{\square\in\lambda}c_\alpha(\square)^r,
 \qquad r\geq0.
\end{equation}
Thus $C_0^{(\alpha)}(\lambda)=|\lambda|$.  We denote by
$\mathbf C_r^{(\alpha)}$ the unique diagonal operator on $\Sym$ such
that
\begin{equation}\label{eq:content-operator-definition}
 \mathbf C_r^{(\alpha)}P_\lambda^{(\alpha)}
 =C_r^{(\alpha)}(\lambda)P_\lambda^{(\alpha)}.
\end{equation}
It is stable because its eigenvalue will be expressed below as a
polynomial in the stable functions $q_2,\ldots,q_{r+2}$.
Define its exponential generating series by
\begin{equation}\label{eq:C-exponential-definition}
 \mathbf C^{(\alpha)}(v)
 =\sum_{r\geq0}\mathbf C_r^{(\alpha)}
   \frac{v^r}{r!}.
\end{equation}

\begin{proposition}\label{prop:content-generating-series}
The stable content Hamiltonians have the generating series
\begin{equation}\label{eq:content-generating-series}
  {
 \mathbf C^{(\alpha)}(v)
 =\frac{\cQ^{\exp}(\alpha v)}
 {(1-e^{-v})(e^{\alpha v}-1)}.}
\end{equation}
Its eigenvalue on $P_\lambda^{(\alpha)}$ is
\begin{equation}\label{eq:Cexp-spectrum}
 C_\lambda^{(\alpha)}(v)
 =\sum_{\square\in\lambda}e^{v c_\alpha(\square)}.
\end{equation}
Although the denominator in \eqref{eq:content-generating-series}
vanishes to order two, the numerator does also, since $\cQ_1=0$ and
$\cQ_2=2\theta E$.
\end{proposition}

\begin{proof}
Summing the geometric progression in each row gives
\begin{align}
 \sum_{\square\in\lambda}e^{v c_\alpha(\square)}
 &=\sum_{i\geq1}\sum_{j=0}^{\lambda_i-1}
   e^{v(\alpha j-(i-1))}\notag\\
 &=\frac1{e^{\alpha v}-1}
   \sum_{i\geq1}
   \left(e^{\alpha vA_i(\lambda)}
        -e^{\alpha vA_i(0)}\right).
 \label{eq:row-geometric-series}
\end{align}
Taking $u=\alpha v$ in \eqref{eq:Qexp-spectrum}, and using
$\theta\alpha=1$, proves the result.
\end{proof}

Subtracting the constant content moment gives the series
\begin{equation}\label{eq:content-zero-mode-series}
 \mathscr Z_0^{(\alpha)}(v)
 =\mathbf C^{(\alpha)}(v)-E
 =\sum_{r\geq1}\mathbf C_r^{(\alpha)}\frac{v^r}{r!}.
\end{equation}
whose eigenvalue is
$\sum_{\square\in\lambda}(e^{vc_\alpha(\square)}-1)$.
Thus \eqref{eq:content-generating-series} constructs all content
Hamiltonians directly from the stable logarithmic series.  Its Hall adjoint
\begin{equation}\label{eq:content-zero-mode-adjoint}
 \mathscr Z_{0,\times}^{(\alpha)}(v)
 =\left(\mathscr Z_0^{(\alpha)}(v)\right)^{\perp_{\Hall}}
\end{equation}
is the corresponding series in the multiplication convention.

It is also useful to record the first operators before taking the
stable limit.  Writing $\cP_m=\cP_m^{(N)}$, direct expansion of
\eqref{eq:PhiN-definition} and \eqref{eq:Q-log-definition} yields
\begin{align}
 \cQ_2^{(N)}&=2\theta\cP_1,
 \label{eq:Q2-P}\\
 \cQ_3^{(N)}&=3\theta
 \bigl(\cP_2-\theta N\cP_1\bigr),
 \label{eq:Q3-P}\\
 \cQ_4^{(N)}&=4\theta
 \bigl(\cP_3-2\theta N\cP_2
       +\theta^2N^2\cP_1\bigr)
       +2\theta^2\cP_1^2.
 \label{eq:Q4-P}
\end{align}
Although $N$ occurs on the right-hand sides, the restrictions of these
combinations to symmetric polynomials are stable by
\cref{prop:Phi-stability}.  Formulas \eqref{eq:Q2-P}--\eqref{eq:Q4-P}
will allow us to compare the content operators both with the
Calogero--Sutherland hierarchy and with the differential operators of
Nazarov and Sklyanin.
\subsection{Normal ordering with the Sergeev--Veselov operator}
\label{sec:SV-normal-ordering}

There is a second construction of the trigonometric CMS integrals which
is particularly well suited to explicit normal ordering.  We describe it
carefully, since from the third integral on it must not be confused with
the Heckman--Polychronakos construction above.

Put $\overline\Lambda=\Lambda[p_0]$ and adjoin an auxiliary variable
$x$.  Following Sergeev and Veselov, define
\begin{equation}\label{eq:SV-Dunkl-infinity}
 \mathscr D_k=\partial-\frac{k}{2}\mathscr A
 \quad\hbox{on }\overline\Lambda[x],
\end{equation}
where $\partial$ is the derivation determined by
\begin{equation}\label{eq:SV-partial}
 \partial(x)=x,
 \qquad
 \partial(p_a)=a x^a \quad(a\geq1),
\end{equation}
and $\mathscr A$ is the $\overline\Lambda$-linear operator determined by
\begin{align}
 \mathscr A(1)&=0,\notag\\
 \mathscr A(x^\ell)
 &=(p_0-2\ell)x^\ell+p_\ell
   +2\sum_{a=1}^{\ell-1}p_a x^{\ell-a}
   \qquad(\ell\geq1).
 \label{eq:SV-A}
\end{align}
Finally let
\begin{equation}\label{eq:SV-evaluation}
 \mathsf E(x^\ell f)=p_\ell f,
 \qquad p_0\text{ being left as a formal parameter},
\end{equation}
and set
\begin{equation}\label{eq:SV-integrals}
 \mathscr L_r(k,p_0)=\mathsf E\mathscr D_k^r\big|_{\overline\Lambda}.
\end{equation}
The specialization $p_0=N$ gives the finite-dimensional trigonometric
CMS integrals of \cite[Section~5]{SergeevVeselov}.

For later verification, we spell out the recursion implicit in these
definitions.  If $F\in\overline\Lambda$ and $\ell\geq1$, then
\begin{align}
 \mathscr D_k(x^\ell F)
={}&x^\ell\sum_{a\geq1}x^aD_aF
 +\left((1+k)\ell-\frac{k p_0}{2}\right)x^\ell F
 \notag\\
&-\frac{k}{2}p_\ell F
-k\sum_{a=1}^{\ell-1}p_a x^{\ell-a}F,
\label{eq:SV-recursion}
\end{align}
For $\ell=0$, one has simply
$\mathscr D_kF=\sum_{a\geq1}x^aD_aF$.  Notice that the first term in
\eqref{eq:SV-recursion} differentiates every $p_a$ occurring in $F$.  Thus
\eqref{eq:SV-recursion} is already a normal-ordering algorithm and no
Leibniz rule may be applied to $\mathscr A$.

For $f\in\overline\Lambda$, the first two iterations give
\begin{align}
 \mathscr D_k f
 &=\sum_{a\geq1}x^aD_af,
 \label{eq:SV-first-iterate}\\
 \mathscr D_k^2 f
={}&\sum_{a,b\geq1}x^{a+b}D_aD_bf
 +\sum_{a\geq1}
   \left((1+k)a-\frac{k p_0}{2}\right)x^aD_af
 \notag\\
&-\frac{k}{2}\sum_{a\geq1}p_aD_af
 -k\sum_{r,s\geq1}x^s p_rD_{r+s}f.
\label{eq:SV-second-iterate}
\end{align}
After applying $\mathsf E$, this becomes
\begin{align}
 \mathscr L_1&=E,\label{eq:SV-L1}\\
 \mathscr L_2
 &={}\sum_{a,b\geq1}p_{a+b}D_aD_b
 -k\sum_{a,b\geq1}p_ap_bD_{a+b}
 \notag\\
 &\quad +(1+k)\sum_{a\geq1}a p_aD_a-kp_0E.
\label{eq:SV-L2}
\end{align}
Consequently, for $k=-\theta$ and $p_0=N$,
\begin{equation}\label{eq:SV-P2-identification}
 \mathscr L_2(-\theta,N)=\cP_2^{(N)}.
\end{equation}
This equality at order two is the source of a possible confusion: it no
longer holds without correction at order three.

\begin{lemma}\label{lem:HP-SV-cubic}
For $k=-\theta$ and $p_0=N$, the restrictions to symmetric
polynomials satisfy
\begin{equation}\label{eq:HP-SV-cubic}
  {
 \cP_3^{(N)}
 =\mathscr L_3(-\theta,N)
  +\frac{\theta N}{2}\mathscr L_2(-\theta,N)
  -\frac{\theta}{2}E^2.}
\end{equation}
\end{lemma}

\begin{proof}
Both sides may be normally ordered in the operators $p_a,D_a$.
For the right-hand side, apply \eqref{eq:SV-recursion} once more to
the four summands of \eqref{eq:SV-second-iterate}, remembering that
$D_a(p_b)=a\delta_{ab}$.  For the left-hand side, expand
$\sum_i(x_i\mathcal D_i)^3$ from \eqref{eq:intro-HP}, restrict to
symmetric polynomials, and pair the terms indexed by $(i,j)$ and
$(j,i)$.  The cubic and quadratic differential blocks agree.  The
remaining contractions are respectively
\begin{equation}
 \frac{\theta N}{2}\mathscr L_2(-\theta,N)
 \quad\text{and}\quad
 -\frac{\theta}{2}
 \left(
   \sum_{a,b\geq1}p_ap_bD_aD_b
   +\sum_{a\geq1}a p_aD_a
 \right).
\end{equation}
The expression in parentheses is $E^2$, which proves
\eqref{eq:HP-SV-cubic}.  Alternatively, the equality can be checked
mechanically from \eqref{eq:SV-recursion}; this is the form used in the
supplementary verification program.
\end{proof}

Substitution in \eqref{eq:Q4-P} removes the apparently extraneous
$E^2$ term and gives the useful stable formula.

\begin{proposition}\label{prop:Q4-SV}
With $p_0=N$,
\begin{equation}\label{eq:Q4-SV}
  {
 \cQ_4
 =4\theta\mathscr L_3(-\theta,p_0)
  -6\theta^2p_0\mathscr L_2(-\theta,p_0)
  +4\theta^3p_0^2E.}
\end{equation}
After normal ordering, the right-hand side is independent of $p_0$.
\end{proposition}

\begin{proof}
Insert \eqref{eq:HP-SV-cubic} and
\eqref{eq:SV-P2-identification} into \eqref{eq:Q4-P}.  This gives
\eqref{eq:Q4-SV} immediately.  Independence of $p_0$ follows already
from stability, but it is also visible directly from
\eqref{eq:SV-recursion}: the coefficients of $p_0^2$ and $p_0$ cancel
separately in the displayed linear combination.
\end{proof}

\section{Content moments and Bernoulli polynomials}
\label{sec:content-Bernoulli}

We now pass from shifted row coordinates to contents.  This step is
elementary, but it is the precise mechanism that connects the
Dunkl--Heckman hierarchy with the class-algebra operators.

Recall the definitions \eqref{eq:alpha-content} and
\eqref{eq:content-moment}.  Our convention is chosen so that
$c_1(\square)=j-i$.

We use Bernoulli polynomials with generating function
\begin{equation}\label{eq:Bernoulli-generating}
 \frac{ue^{xu}}{e^u-1}
 =\sum_{m\geq0}B_m(x)\frac{u^m}{m!},
\end{equation}
so that $B_1(x)=x-\frac12$ and
\begin{equation}\label{eq:Bernoulli-sums}
 \sum_{j=0}^{n-1}(x+j)^r
 =\frac{B_{r+1}(x+n)-B_{r+1}(x)}{r+1}.
\end{equation}

\begin{proposition}\label{prop:content-Bernoulli}
With the shifted row coordinates $A_i(\lambda)$ of
\eqref{eq:A-coordinates}, one has
\begin{equation}\label{eq:content-Bernoulli}
  {
 C_r^{(\alpha)}(\lambda)
 =\frac{\alpha^r}{r+1}
 \sum_{i\geq1}
 \left[
 B_{r+1}\bigl(A_i(\lambda)\bigr)
 -B_{r+1}\bigl(A_i(0)\bigr)
 \right].}
\end{equation}
In particular, the content moments belong to the shifted-symmetric
algebra generated by the $q_s$.
\end{proposition}

\begin{proof}
The contribution of row $i$ to \eqref{eq:content-moment} is
\begin{align}
 \sum_{j=1}^{\lambda_i}
 \bigl(\alpha(j-1)-(i-1)\bigr)^r
 &=\alpha^r\sum_{j=0}^{\lambda_i-1}
 \left(j-\frac{i-1}{\alpha}\right)^r\\
 &=\frac{\alpha^r}{r+1}
 \left[
 B_{r+1}\left(\lambda_i-\frac{i-1}{\alpha}\right)
 -B_{r+1}\left(-\frac{i-1}{\alpha}\right)
 \right],
\end{align}
by \eqref{eq:Bernoulli-sums}.  Since $\theta=\alpha^{-1}$, the two
arguments in the last line are precisely $A_i(\lambda)$ and $A_i(0)$.
Summing over the rows proves \eqref{eq:content-Bernoulli}.

The difference of Bernoulli polynomials in
\eqref{eq:content-Bernoulli} is a linear combination of the shifted
row power sums $S_1,\ldots,S_{r+1}$.  By
\eqref{eq:q-vs-S-general}, these generate the same algebra as
$q_2,\ldots,q_{r+2}$.
\end{proof}

The preceding relation is triangular.  More precisely, the coefficient
of $q_{r+2}$ in $C_r^{(\alpha)}$ is
\begin{equation}\label{eq:content-leading-q}
 \frac{\alpha^{r+1}}{(r+1)(r+2)}.
\end{equation}
Indeed, the leading term of $B_{r+1}(x)$ is $x^{r+1}$, while
$q_{r+2}=(r+2)\theta S_{r+1}$ modulo
$S_1,\ldots,S_r$.  Thus either the content moments or the logarithmic
Hamiltonians may be used as algebraically independent generators.

\subsection{The first content moments}

Using
\begin{equation}
 B_2(x)=x^2-x+\frac16,\qquad
 B_3(x)=x^3-\frac32x^2+\frac12x,
\end{equation}
together with \eqref{eq:q2-S}--\eqref{eq:q4-S}, we obtain
\begin{align}
 C_0^{(\alpha)}
 &=\frac{\alpha}{2}q_2,
 \label{eq:C0-q}\\
 C_1^{(\alpha)}
 &=\frac{\alpha^2}{6}q_3
 +\frac{\alpha(1-\alpha)}4q_2,
 \label{eq:C1-q}\\
 C_2^{(\alpha)}
 &=\frac{\alpha^3}{12}q_4
 +\frac{\alpha^2(1-\alpha)}6q_3
 +\frac{\alpha(\alpha^2-3\alpha+1)}{12}q_2.
 \label{eq:C2-q}
\end{align}
These identities are identities of shifted-symmetric functions, not
merely evaluations for partitions of bounded size.

For the operators defined by \eqref{eq:content-operator-definition},
equations \eqref{eq:C0-q}--\eqref{eq:C2-q} give
\begin{align}
 \mathbf C_0^{(\alpha)}
 &=\frac{\alpha}{2}\cQ_2=E,
 \label{eq:C0-operator}\\
 \mathbf C_1^{(\alpha)}
 &=\frac{\alpha^2}{6}\cQ_3
 +\frac{\alpha(1-\alpha)}4\cQ_2,
 \label{eq:C1-operator}\\
\mathbf C_2^{(\alpha)}
&=\frac{\alpha^3}{12}\cQ_4
 +\frac{\alpha^2(1-\alpha)}6\cQ_3
 +\frac{\alpha(\alpha^2-3\alpha+1)}{12}\cQ_2.
 \label{eq:C2-operator}
\end{align}

\subsection{Elementary and complete functions of the content alphabet}
\label{sec:elementary-complete-contents}

The preceding moments allow us to apply an arbitrary symmetric function
to the content alphabet
\begin{equation}
 C_\alpha(\lambda)=
 \{c_\alpha(\square)\mid\square\in\lambda\}.
\end{equation}
Two particularly natural families are defined by
\begin{align}
 \mathsf E_\lambda(u)
 &=\sum_{k\geq0}e_k[C_\alpha(\lambda)]u^k
   =\prod_{\square\in\lambda}
    (1+u c_\alpha(\square)),
 \label{eq:content-elementary-eigenvalue}\\
 \mathsf H_\lambda(u)
 &=\sum_{k\geq0}h_k[C_\alpha(\lambda)]u^k
   =\prod_{\square\in\lambda}
    \frac1{1-u c_\alpha(\square)}.
 \label{eq:content-complete-eigenvalue}
\end{align}
The box $(1,1)$ has zero content, so the first product has degree at
most $|\lambda|-1$ when $\lambda$ is nonempty.

Since the operators $\mathbf C_r^{(\alpha)}$ commute, the usual
power-sum formulas define unambiguously two operator series
\begin{align}
 \mathsf E^{(\alpha)}(u)
 &=\exp\left(
   \sum_{r\geq1}\frac{(-1)^{r-1}}r
   u^r\mathbf C_r^{(\alpha)}\right),
 \label{eq:content-elementary-operator}\\
 \mathsf H^{(\alpha)}(u)
 &=\exp\left(
   \sum_{r\geq1}\frac1r
   u^r\mathbf C_r^{(\alpha)}\right).
 \label{eq:content-complete-operator}
\end{align}
Their eigenvalues are respectively
\eqref{eq:content-elementary-eigenvalue} and
\eqref{eq:content-complete-eigenvalue}, and
\begin{equation}
 \mathsf E^{(\alpha)}(-u)\mathsf H^{(\alpha)}(u)=1.
\end{equation}
Writing $\mathsf E^{(\alpha)}(u)=\sum_{k\geq0}\mathsf E_k^{(\alpha)}u^k$
and similarly for $\mathsf H^{(\alpha)}$, one finds
\begin{align}
 \mathsf E_1^{(\alpha)}&=\mathbf C_1^{(\alpha)},&
 \mathsf H_1^{(\alpha)}&=\mathbf C_1^{(\alpha)},\notag\\
 \mathsf E_2^{(\alpha)}
 &=\frac12\left((\mathbf C_1^{(\alpha)})^2
                 -\mathbf C_2^{(\alpha)}\right),&
 \mathsf H_2^{(\alpha)}
 &=\frac12\left((\mathbf C_1^{(\alpha)})^2
                 +\mathbf C_2^{(\alpha)}\right),
 \label{eq:first-EH-content}\\
 \mathsf E_3^{(\alpha)}
 &=\frac16\left((\mathbf C_1^{(\alpha)})^3
 -3\mathbf C_1^{(\alpha)}\mathbf C_2^{(\alpha)}
 +2\mathbf C_3^{(\alpha)}\right),&
 \mathsf H_3^{(\alpha)}
 &=\frac16\left((\mathbf C_1^{(\alpha)})^3
 +3\mathbf C_1^{(\alpha)}\mathbf C_2^{(\alpha)}
 +2\mathbf C_3^{(\alpha)}\right).
 \notag
\end{align}
Thus \eqref{eq:C1-operator}--\eqref{eq:C2-operator} already give
explicit formulas for the first two members in terms of the
$\cQ_r$.  The multiplication-convention operators are their Hall
adjoints.  In particular,
\begin{equation}\label{eq:E1-content-Delta2}
 (\mathsf E_1^{(\alpha)})^{\perp_{\Hall}}
 =(\mathsf H_1^{(\alpha)})^{\perp_{\Hall}}
 =\frac12\Delta_2(\alpha).
\end{equation}

\subsubsection{The classical specialization and monotone Hurwitz numbers}

At $\alpha=1$, the operators above are symmetric functions of the
Jucys--Murphy elements.  Jucys' identity
\begin{equation}\label{eq:Jucys-elementary-identity}
 \prod_{i=1}^n(v+\xi_i)
 =\sum_{\sigma\in\mathfrak S_n}
   v^{\#\operatorname{cycles}(\sigma)}\sigma
\end{equation}
shows that $e_k(\Xi_n)$ is the sum of all permutations having $n-k$
cycles.  Complete functions are subtler.  Following
\cite{LassalleClassExpansion,FerayComplete}, define
\begin{equation}\label{eq:monotone-Phi-definition}
 \Phi(z,u)=\mathsf H^{(1)}(u)e^{zp_1}
 =\sum_{n\geq0}\Phi_n(u)\frac{z^n}{n!},
\end{equation}
where
\begin{equation}\label{eq:monotone-Phi-Schur}
 \Phi_n(u)=\sum_{\lambda\vdash n}
 f^\lambda\prod_{\square\in\lambda}
 \frac1{1-u c(\square)}s_\lambda.
\end{equation}
The coefficient of $u^k$ encodes $h_k(\Xi_n)$.

We include a short derivation of its evolution equation.  If
$\lambda=\mu\setminus\square$, the branching rule and
\eqref{eq:monotone-Phi-Schur} give
\begin{equation}
 \langle p_1\Phi_{n-1},s_\mu\rangle_{\Hall}
 =\sum_{\mu/\lambda=\square}
 \frac{f^\lambda}{f^\mu}(1-u c(\square))
 \langle\Phi_n,s_\mu\rangle_{\Hall}.
 \label{eq:Phi-branching-step}
\end{equation}
We shall use the following standard content identity.

\begin{lemma}\label{lem:removable-content-average}
For every $\mu\vdash n$,
\begin{equation}
 \sum_{\mu/\lambda=\square}
 \frac{f^\lambda}{f^\mu}c(\square)
 =\frac2n\sum_{\square\in\mu}c(\square).
 \label{eq:removable-content-average}
\end{equation}
\end{lemma}

\begin{proof}
Let $L_{-1}=[\mathscr G,p_1]=\sum_{j\geq1}p_{j+1}D_j$ and let
$L_1=L_{-1}^{\perp_{\Hall}}$.  Pieri's rule identifies the left-hand
side of \eqref{eq:removable-content-average}, multiplied by $f^\mu$,
with
\begin{equation}
 \langle L_1s_\mu,p_1^{n-1}\rangle_{\Hall}
 =\langle s_\mu,L_{-1}p_1^{n-1}\rangle_{\Hall}
 =(n-1)\langle s_\mu,p_2p_1^{n-2}\rangle_{\Hall}.
\end{equation}
The Frobenius character formula for a transposition gives
\begin{equation}
 \frac{n(n-1)}2\frac{1}{f^\mu}
 \langle s_\mu,p_2p_1^{n-2}\rangle_{\Hall}
 =\sum_{\square\in\mu}c(\square),
\end{equation}
which proves the claim.
\end{proof}

Since $\sum_{\mu/\lambda=\square}f^\lambda=f^\mu$, equations
\eqref{eq:Phi-branching-step} and
\eqref{eq:removable-content-average} give the following result.

\begin{theorem}\label{thm:monotone-evolution}
The series \eqref{eq:monotone-Phi-definition} is the unique solution of
\begin{equation}\label{eq:monotone-evolution}
 z\frac{\partial\Phi}{\partial z}
 =zp_1\Phi+2u\mathscr G\Phi,
 \qquad \Phi(0,u)=1.
\end{equation}
\end{theorem}

Indeed, comparison of the coefficient of $z^n/n!$ reduces
\eqref{eq:monotone-evolution} precisely to the two preceding formulas.
The logarithm of \eqref{eq:monotone-Phi-definition} is the generating
series of connected simple monotone Hurwitz numbers
\cite{GGPN13}.  Equation \eqref{eq:monotone-evolution} therefore yields
their monotone cut-and-join equation; see also \cite{DKPS}.  Lassalle's
recurrences and the Virasoro constraints of F\'eray are alternative
forms of the same structure.  The $b=\alpha-1$ deformation of this
picture, including an evolution equation and Virasoro constraints, is
developed by Bonzom, Chapuy, and Do{\l}\k{e}ga
\cite{BonzomChapuyDolega}.  

\subsection{The transposition observable}

In our reduced-cycle normalization, the spectral Hamiltonian associated
with the transposition class has eigenvalue
$2C_1^{(\alpha)}(\lambda)$.  We are therefore led to
\begin{equation}\label{eq:H2-from-content}
  {
 \cH_2(\alpha)
 =2\mathbf C_1^{(\alpha)}
 =\frac{\alpha^2}{3}\cQ_3
 +\frac{\alpha(1-\alpha)}2\cQ_2.}
\end{equation}
At $\alpha=1$, this is twice the ordinary sum-of-contents operator, in
agreement with \eqref{eq:Delta2-vs-Goulden}.

\section{The operator $\Delta_2(\alpha)$}
\label{sec:Delta2}

We now return from the spectral convention to the multiplication
convention of the Goulden--Jackson product.  Recall that, in the latter
convention, $\Delta_\mu(\alpha)$ denotes multiplication by the stable
series $\widehat p_\mu$.  The Hall-dual Jack functions ${J'}_\lambda^{(\alpha)}$ diagonalize this product,
but our identification with the ordinary Jack basis reverses the two
arguments of the Hall pairing.  Accordingly, as stipulated in
\eqref{eq:adjoint-convention},
\begin{equation}\label{eq:Delta-H-adjoint-again}
 \Delta_\mu(\alpha)=\cH_\mu(\alpha)^{\perp_{\Hall}}.
\end{equation}

For partitions, or finite multisets, $I$ and $J$ of positive integers,
we use the normally ordered notation
\begin{equation}\label{eq:block-notation}
 (I\mid J)=p_I D_{p_J}.
\end{equation}
Since multiplication by $p_k$ and $D_k$ are Hall adjoints,
\begin{equation}\label{eq:block-Hall-adjoint}
 (I\mid J)^{\perp_{\Hall}}=(J\mid I).
\end{equation}
Thus Hall duality amounts simply to reflecting every differential
block.

\subsection{The second operator}

We first make the comparison with the finite-variable Hamiltonian
explicit.  On symmetric polynomials, the second
Heckman--Polychronakos operator is
\begin{equation}\label{eq:P2-power-sums}
\begin{split}
 \cP_2^{(N)}={}&
 \sum_{i,j\geq1}p_{i+j}D_iD_j
 +\theta\sum_{i,j\geq1}p_ip_jD_{i+j}\\
 &+\sum_{k\geq1}\bigl((1-\theta)k+\theta N\bigr)p_kD_k.
\end{split}
\end{equation}
All sums are locally finite on $\Lambda_N$.  This is the usual
collective-variable form of the Calogero--Sutherland Hamiltonian.

\begin{theorem}\label{thm:Delta2-explicit}
In the conventions of this paper,
\begin{equation}\label{eq:Delta2-explicit-new}
 {
\begin{split}
 \Delta_2(\alpha)={}&
 \alpha\sum_{i,j\geq1}p_ip_jD_{i+j}
 +\sum_{i,j\geq1}p_{i+j}D_iD_j\\
 &+(\alpha-1)\sum_{k\geq1}(k-1)p_kD_k.
\end{split}}
\end{equation}
Its eigenvalue, in the dual Jack convention, is
$2C_1^{(\alpha)}(\lambda)$.
\end{theorem}

\begin{proof}
Substituting \eqref{eq:Q2-P} and \eqref{eq:Q3-P} into
\eqref{eq:H2-from-content}, and using $\theta=\alpha^{-1}$, gives
\begin{align}
 \cH_2(\alpha)
 &=\alpha\bigl(\cP_2^{(N)}-\theta N\cP_1^{(N)}\bigr)
 +(1-\alpha)\cP_1^{(N)}\notag\\
 &=\alpha\cP_2^{(N)}-(N+\alpha-1)E.
 \label{eq:H2-vs-P2}
\end{align}
Since $\cP_1^{(N)}=E$, formula \eqref{eq:P2-power-sums} yields
\begin{equation}\label{eq:H2-explicit}
\begin{split}
 \cH_2(\alpha)={}&
 \sum_{i,j\geq1}p_ip_jD_{i+j}
 +\alpha\sum_{i,j\geq1}p_{i+j}D_iD_j\\
 &+(\alpha-1)\sum_{k\geq1}(k-1)p_kD_k.
\end{split}
\end{equation}
In particular, all dependence on $N$ has disappeared.  Applying
\eqref{eq:block-Hall-adjoint} to \eqref{eq:H2-explicit} gives
\eqref{eq:Delta2-explicit-new}.  The assertion about the eigenvalue is
\eqref{eq:H2-from-content}.
\end{proof}

At $\alpha=1$, formula \eqref{eq:Delta2-explicit-new} is twice
Goulden's operator \eqref{eq:classical-cutjoin}.  The last line of
\eqref{eq:Delta2-explicit-new} is the first genuinely Jack correction;
it is diagonal and vanishes in the Schur case.

\section{The third stable class operator
\texorpdfstring{$\Delta_3(\alpha)$}{Delta3(alpha)}}
\label{sec:Delta3}

We now apply the preceding normal-ordering algorithm to the next
content observable.  In the reduced-cycle normalization used for the
stable class product, the shifted power sum attached to a $3$-cycle is
\begin{equation}\label{eq:p3sharp-content}
 p_3^\#(\lambda;\alpha)
 =3\left(
 C_2^{(\alpha)}(\lambda)
 -(\alpha-1)C_1^{(\alpha)}(\lambda)
 -\alpha\binom{|\lambda|}{2}
 \right).
\end{equation}
Accordingly, define the spectral Hamiltonian
\begin{equation}\label{eq:H3-content-definition}
 \cH_3(\alpha)
 =3\left(
 \mathbf C_2^{(\alpha)}
 -(\alpha-1)\mathbf C_1^{(\alpha)}
 -\frac{\alpha}{2}(E^2-E)
 \right).
\end{equation}
Using \eqref{eq:C0-operator}--\eqref{eq:C2-operator}, one obtains
\begin{align}
 \cH_3(\alpha)={}&
 \frac{\alpha^3}{4}\cQ_4
 +\alpha^2(1-\alpha)\cQ_3
 \notag\\
 &+\frac{\alpha(4\alpha^2-9\alpha+4)}4\cQ_2
 -\frac{3\alpha}{2}(E^2-E).
\label{eq:H3-Q}
\end{align}

Formula \eqref{eq:Q4-SV} turns this expression into a direct normal
ordering problem.  Before stating the result, it is useful to record
the complete cancellation of the auxiliary dimension.

\begin{lemma}\label{lem:H3-SV}
In $\operatorname{End}(\Lambda[p_0])$ one has
\begin{align}
 \cH_3(\alpha)
={}&\alpha^2\mathscr L_3(-\alpha^{-1},p_0)
 +\left(
  3\alpha(1-\alpha)-\frac{3\alpha p_0}{2}
  \right)
  \mathscr L_2(-\alpha^{-1},p_0)
 \notag\\
&+\left(
 p_0^2+3(\alpha-1)p_0+2\alpha^2-3\alpha+2
 \right)E
 -\frac{3\alpha}{2}E^2.
\label{eq:H3-SV}
\end{align}
The normally ordered expression on the right is independent of $p_0$.
\end{lemma}

\begin{proof}
Substitute \eqref{eq:Q2-P}, \eqref{eq:Q3-P}, and
\eqref{eq:Q4-SV} into \eqref{eq:H3-Q}, and use
$\theta=\alpha^{-1}$ and $\mathscr L_1=E$.  Collecting the
coefficients of $\mathscr L_3,\mathscr L_2,E,E^2$ gives
\eqref{eq:H3-SV}.  To check the last assertion without using the
spectral construction, insert \eqref{eq:SV-L2} and compute
$\mathscr L_3=\mathsf E\mathscr D_{-1/\alpha}^3$ by
\eqref{eq:SV-recursion}.  The coefficients of $p_0^2$ cancel between
the first and third terms, and those of $p_0$ cancel between the first
three terms.  No specialization of $p_0$ is used.
\end{proof}

For complete reproducibility, we record the output of the normal
ordering before taking the Hall adjoint.  Applying
\eqref{eq:SV-recursion} to \eqref{eq:H3-SV} gives
\begin{align}
\cH_3(\alpha)
={}&\sum_{i,j,k\geq1}\Bigl(
 p_ip_jp_kD_{i+j+k}
 +\alpha p_{i+j}p_kD_{i+k}D_j
 \notag\\
&\qquad
 +\alpha p_ip_{j+k}D_{i+j}D_k
 +p_{i+j+k}D_{i+j+k}
 \notag\\
&\qquad
 +\alpha^2p_{i+j+k}D_iD_jD_k
 +\alpha p_{i+k}p_jD_iD_{j+k}
 \Bigr)
 \notag\\
&+(\alpha-1)\Biggl[
 \frac32\sum_{i,j\geq1}(i+j-2)
 \Bigl(
  \alpha p_{i+j}D_iD_j
  +p_ip_jD_{i+j}
 \Bigr)
 \notag\\
&\hspace{35mm}
 +(2\alpha-1)\sum_{k\geq1}
 \binom{k-1}{2}p_kD_k
 \Biggr].
\label{eq:H3-normal-ordered}
\end{align}
Each line can therefore be checked directly from
\eqref{eq:SV-recursion}, without using the eigenvalues or the proposed
formula for the multiplication operator.

\begin{theorem}\label{thm:Delta3-explicit}
The multiplication operator associated with the stable reduced
$3$-cycle is
\begin{align}
\Delta_3(\alpha)
={}&\sum_{i,j,k\geq1}\Bigl(
 p_{i+j+k}D_iD_jD_k
 +\alpha p_{i+k}p_jD_{i+j}D_k
 \notag\\
&\qquad
 +\alpha p_{i+j}p_kD_iD_{j+k}
 +p_{i+j+k}D_{i+j+k}
 \notag\\
&\qquad
 +\alpha^2p_ip_jp_kD_{i+j+k}
 +\alpha p_ip_{j+k}D_{i+k}D_j
 \Bigr)
 \notag\\
&+(\alpha-1)\Biggl[
 \frac32\sum_{i,j\geq1}(i+j-2)
 \Bigl(
  \alpha p_ip_jD_{i+j}
  +p_{i+j}D_iD_j
 \Bigr)
 \notag\\
&\hspace{35mm}
 +(2\alpha-1)\sum_{k\geq1}
 \binom{k-1}{2}p_kD_k
 \Biggr].
\label{eq:Delta3-explicit}
\end{align}
Its eigenvalue in the dual Jack convention is
$p_3^\#(\lambda;\alpha)$ from \eqref{eq:p3sharp-content}.
\end{theorem}

\begin{proof}
Apply \eqref{eq:SV-recursion} successively three times.  In the third
iteration there are four types of contributions:
\begin{enumerate}
 \item $\partial$ differentiates the final argument three times;
 \item it differentiates once a power sum created by $\mathscr A$;
 \item $\mathscr A$ is used once or twice;
 \item the scalar terms $-2\ell x^\ell$ in \eqref{eq:SV-A} contract
       one or two indices.
\end{enumerate}
After applying $\mathsf E$, substitute the result in
\eqref{eq:H3-SV}.  The first three types give the six terms in the
triple sum of \eqref{eq:H3-normal-ordered}.  The
one-index contractions give
\begin{equation}
 (\alpha-1)(2\alpha-1)
 \sum_{k\geq1}\binom{k-1}{2}p_kD_k,
\end{equation}
while the two-index contractions give
\begin{equation}
 \frac32(\alpha-1)
 \sum_{i,j\geq1}(i+j-2)
 \left(
  \alpha p_ip_jD_{i+j}+p_{i+j}D_iD_j
 \right).
\end{equation}
Together these contributions give \eqref{eq:H3-normal-ordered}.
Its normally ordered Hall adjoint is exactly
\eqref{eq:Delta3-explicit}.  Finally,
$\Delta_3(\alpha)=\cH_3(\alpha)^{\perp_{\Hall}}$ by
\eqref{eq:Delta-H-adjoint-again}, and \eqref{eq:p3sharp-content}
gives the stated eigenvalue.
\end{proof}

\section{Integral normal forms and polynomiality}
\label{sec:integral-normal-forms}

The preceding calculation suggests an additional use of the
Sergeev--Veselov recursion.  Although
$\mathscr D_{-1/\alpha}$ is written with the denominator $\alpha$,
stabilization and Hall duality leave only polynomial coefficients in
$\alpha$.  This phenomenon is already visible in
\eqref{eq:Delta2-explicit-new} and \eqref{eq:Delta3-explicit}.

To state the integrality correctly, sums over ordered indices should be
grouped into their orbits.  For example, when $i\ne j$, the two ordered
pairs $(i,j)$ and $(j,i)$ have the same normally ordered block.  Hence
the apparent factor $3/2$ in \eqref{eq:Delta3-explicit} contributes
$3(i+j-2)$ on such an orbit; when $i=j$, the factor $i+j-2$ is even.
Thus all orbit coefficients in \eqref{eq:Delta3-explicit} belong to
$\mathbb Z[\alpha]$.

\begin{proposition}\label{prop:integrality-first-operators}
After collecting identical normally ordered blocks, every coefficient
of $\Delta_2(\alpha)$ and $\Delta_3(\alpha)$ belongs to
$\mathbb Z[\alpha]$.
\end{proposition}

\begin{proof}
For $\Delta_2$ this is immediate from
\eqref{eq:Delta2-explicit-new}.  For $\Delta_3$, the six triple sums
have coefficients in $\mathbb Z[\alpha]$.  The preceding orbit
argument removes the only apparent denominator in the two-index term,
and $\binom{k-1}{2}$ is integral in the one-index term.
\end{proof}

More generally, iteration of \eqref{eq:SV-recursion} produces only
integer multiplicities before the substitutions
$k=-\alpha^{-1}$ and the stable normalization are made.  It therefore
reduces polynomiality to two precise questions: cancellation of the
formal dimension $p_0$, and cancellation of the negative powers of
$\alpha$.  The first cancellation follows spectrally for every stable
content Hamiltonian.  The second is proved above in the first two
nontrivial cases, but a uniform integral normalization of all higher
Hamiltonians remains to be established.

\begin{conjecture}\label{conj:integral-normal-form}
For every reduced cycle type $\mu$, the stable multiplication operator
$\Delta_\mu(\alpha)$ has, after collecting identical normally ordered
blocks, coefficients in $\mathbb Z[\alpha]$.
\end{conjecture}

A polynomiality theorem for the stable structure constants of Jack
characters is already known.  Do\l\k{e}ga and F\'eray prove in
\cite[Theorem~1.4]{DolegaFerayGaussian} that, in their normalization,
the structure constants are polynomials in
$\gamma=\alpha^{-1/2}-\alpha^{1/2}$, with explicit parity and degree
bounds.  Thus \cref{conj:integral-normal-form} is not meant as a new
polynomiality statement for the abstract stable product.  It asks for
a compatible integral normal form of the multiplication operators, and
for a direct explanation of polynomiality by the
Sergeev--Veselov normal-ordering algorithm.  Comparing precisely the two
normalizations should clarify which part of the known theorem is already
visible at the operator level.

This assertion is weaker than the $b$-conjecture.  It predicts neither
positivity in $b=\alpha-1$ nor a combinatorial interpretation of the
coefficients.  It would nevertheless give an operator-theoretic route
to the polynomiality of the Jack structure coefficients.  To deduce
that statement from \cref{conj:integral-normal-form}, one must in
addition control integrally the triangular change of basis between the
stable class series and the content Hamiltonians; normal ordering alone
does not supply this last step.

\section{From the Jack Hamiltonian to the affine Yangian}
\label{sec:Jack-Virasoro}

We now return to the question that motivates the comparison with
Frenkel and Wang.  In the Schur case, the commutators of the
cut-and-join operator with the Heisenberg modes are precisely the
Virasoro modes, and further commutators generate
$\mathcal W_{1+\infty}$.  For general $\alpha$, the first commutators
remain quadratic, but their linear terms depend on the mode.  They have
a familiar Feigin--Fuchs interpretation; this gives a Virasoro
completion, but not the sought deformation of the Frenkel--Wang
construction.  The latter is supplied by the affine Yangian of
$\mathfrak{gl}_1$.

\subsection{The direct commutator calculation}

Recall our Jack Heisenberg convention
\begin{equation}
 a_{-k}=\alpha p_k,\qquad a_k=D_k,\qquad
 [a_k,a_{-l}]=\alpha k\delta_{kl}\quad(k,l\geq1).
\end{equation}
A direct calculation from \eqref{eq:Delta2-explicit-new} gives, for
$k\geq1$,
\begin{align}
 [\Delta_2(\alpha),a_{-k}]
 ={}&k\sum_{i=1}^{k-1}a_{-i}a_{-(k-i)}
 +2k\sum_{j\geq1}a_{-(k+j)}a_j
 \notag\\
 &+(\alpha-1)k(k-1)a_{-k},
 \label{eq:Delta-a-minus}\\
 [\Delta_2(\alpha),a_k]
 ={}&-k\sum_{i=1}^{k-1}a_i a_{k-i}
 -2k\sum_{j\geq1}a_{-j}a_{k+j}
 \notag\\
 &-(\alpha-1)k(k-1)a_k.
 \label{eq:Delta-a-plus}
\end{align}
Let
\begin{equation}\label{eq:Jack-Sugawara}
 L_n^{(0)}=\frac1{2\alpha}
 \sum_{r\in\mathbb Z}:a_{n-r}a_r:,
 \qquad a_0=0.
\end{equation}
These are the ordinary Sugawara modes for the Heisenberg algebra
\begin{equation}
 [a_m,a_n]=\alpha m\delta_{m,-n}.
\end{equation}
They satisfy
\begin{align}
 [L_m^{(0)},a_n]&=-n a_{m+n},
 \label{eq:Sugawara-Heisenberg}\\
 [L_m^{(0)},L_n^{(0)}]
 &=(m-n)L_{m+n}^{(0)}
 +\frac{m^3-m}{12}\delta_{m,-n}.
 \label{eq:Sugawara-Virasoro}
\end{align}
For negative indices, \eqref{eq:Delta-a-minus} can be rewritten as
\begin{equation}\label{eq:negative-FF-from-Delta}
 \frac1{2\alpha k}[\Delta_2(\alpha),a_{-k}]
 =L_{-k}^{(0)}+Q_\alpha(1-k)a_{-k},
 \qquad
 Q_\alpha=-\frac{\alpha-1}{2\alpha}.
\end{equation}
Thus this half of the commutator calculation has precisely the
Feigin--Fuchs dependence $Q_\alpha(n+1)a_n$, with $n=-k$.  After the
usual rescaling of the free boson, it extends to the standard
background-charge realization of Virasoro.  The numerical expression
for the central charge depends on the normalization of the boson and
of the Jack parameter; in the convention of Awata--Matsuo--Odake--
Shiraishi it is
\begin{equation}\label{eq:Awata-central-charge}
 c=1-\frac{6(1-\beta)^2}{\beta}.
\end{equation}
Their operators with indices $n\geq-1$ are constraints on a
Selberg--Aomoto partition function, and
the Hamiltonian is expressed
as a normally ordered sum of a creation mode times a Virasoro mode,
up to a momentum term \cite{AwataEtAl}.

The positive-index expression obtained directly from
\eqref{eq:Delta-a-plus} contains $k-1$ rather than $k+1$.  Hence the two
halves extracted from the same commutator prescription do not form one
Feigin--Fuchs family.  Completing \eqref{eq:negative-FF-from-Delta} by
the abstract Feigin--Fuchs formula is legitimate, but it does not
generalize the Frenkel--Wang mechanism, whose purpose is to generate
the ambient operator algebra from the Hamiltonian and Heisenberg modes.

\subsection{The affine-Yangian completion}

The correct completion is visible in the Drinfeld presentation of the
affine Yangian of $\mathfrak{gl}_1$.  Let
\begin{equation}\label{eq:Yangian-currents}
 e(z)=\sum_{r\geq0}e_rz^{-r-1},\qquad
 f(z)=\sum_{r\geq0}f_rz^{-r-1},\qquad
 \psi(z)=1+\sigma_3\sum_{r\geq0}\psi_rz^{-r-1},
\end{equation}
where $\sigma_3=h_1h_2h_3$.  The coefficients of $\psi(z)$ form the
commuting Cartan family.

In the standard Drinfeld normalization, the defining relations include
\begin{align}
 [e_i,f_j]&=\psi_{i+j},\label{eq:Yangian-ef}\\
 [\psi_3,e_j]&=6e_{j+1}+2\sigma_3\psi_0e_j,
 \label{eq:Yangian-psi3-e}\\
 [\psi_3,f_j]&=-6f_{j+1}-2\sigma_3\psi_0f_j.
 \label{eq:Yangian-psi3-f}
\end{align}
See \cite{Tsymbaliuk}; the translation between the Yangian generators,
the Maulik--Okounkov $R$-matrix and the Nazarov--Sklyanin Hamiltonians
is made explicit in \cite{Prochazka}.  Consequently,
\begin{align}
 e_{j+1}&=\frac16\bigl([\psi_3,e_j]
              -2\sigma_3\psi_0e_j\bigr),
 \label{eq:Yangian-e-recursion}\\
 f_{j+1}&=-\frac16\bigl([\psi_3,f_j]
              +2\sigma_3\psi_0f_j\bigr).
 \label{eq:Yangian-f-recursion}
\end{align}
These formulas are the deformed analogue of the commutator recursion
used by Frenkel and Wang.  They also explain why a search restricted to
quadratic Virasoro operators is too small: for generic $\alpha$, the
successive raising and lowering operators satisfy the nonlinear
affine-Yangian relations.

We now fix the normalization in our Fock representation.  With
\eqref{eq:Yangian-parameters}, one has
\begin{equation}\label{eq:sigma3-specialized}
 \sigma_3=\alpha(\alpha-1).
\end{equation}
Introduce the rescaled oscillators
\begin{equation}\label{eq:Prochazka-bosons}
 b_{-k}=p_k=\frac1\alpha a_{-k},\qquad
 b_k=\frac1\alpha D_k=\frac1\alpha a_k,
 \qquad k\geq1.
\end{equation}
They satisfy
\begin{equation}
 [b_k,b_{-l}]=\frac{k}{\alpha}\delta_{kl}
 =-\frac{k}{h_1h_2}\delta_{kl},
\end{equation}
which is the one-boson convention used in
\cite[Appendix~B]{Prochazka}.

\begin{proposition}\label{prop:first-Yangian-modes}
In the multiplication convention of this paper, the first Yangian
modes act on $\Sym$ as
\begin{align}
 e_0&=p_1,&
 f_0&=-\frac1\alpha D_1,&
 \psi_0&=\frac1\alpha,&
 \psi_1&=0,\label{eq:first-Yangian-zero-modes}\\
 \psi_2&=2E,&
 \psi_3&=3\Delta_2(\alpha)+2(\alpha-1)E,
 \label{eq:psi3-Delta2-exact}\\
 e_1&=\sum_{j\geq1}p_{j+1}D_j,&
 f_1&=-\sum_{j\geq1}p_jD_{j+1}.
 \label{eq:first-Yangian-ladder-modes}
\end{align}
Equivalently, on the idempotent basis diagonalizing the deformed class
product, the eigenvalue of $\psi_3$ is
\begin{equation}\label{eq:psi3-eigenvalue}
 6C_1^{(\alpha)}(\lambda)+2(\alpha-1)|\lambda|.
\end{equation}
On the spectral side, the Hall adjoint of
\eqref{eq:psi3-Delta2-exact} gives the corresponding formula with
$\cH_2(\alpha)$ in place of $\Delta_2(\alpha)$.
\end{proposition}

\begin{proof}
The normalization $e_0=b_{-1}$ and $f_0=-b_1$ is
\cite[(4.26)]{Prochazka}.  Hence
\begin{equation}
 [e_0,f_0]=\frac1\alpha=\psi_0.
\end{equation}
The same one-boson realization gives
\begin{equation}
 \psi_2=-2h_1h_2\sum_{j\geq1}b_{-j}b_j=2E
\end{equation}
and the expressions in
\eqref{eq:first-Yangian-ladder-modes}; see
\cite[Appendix~B]{Prochazka}.  They may also be recovered directly from
our cut-and-join operator.  Indeed, the case $k=1$ of
\eqref{eq:Delta-a-minus} gives
\begin{equation}\label{eq:Delta-e0-low}
 [\Delta_2(\alpha),e_0]
 =2\sum_{j\geq1}p_{j+1}D_j=2e_1.
\end{equation}
Since $[E,e_0]=e_0$ and
$\sigma_3\psi_0=\alpha-1$, equation
\eqref{eq:psi3-Delta2-exact} yields
\begin{equation}
 [\psi_3,e_0]=6e_1+2\sigma_3\psi_0e_0,
\end{equation}
which is \eqref{eq:Yangian-psi3-e} for $j=0$.  Similarly,
\eqref{eq:Delta-a-plus} at $k=1$ gives
\begin{equation}
 [\Delta_2(\alpha),f_0]=-2f_1,
\end{equation}
and $[E,f_0]=-f_0$, whence
\begin{equation}
 [\psi_3,f_0]=-6f_1-2\sigma_3\psi_0f_0.
\end{equation}
This proves the asserted normalization of $\psi_3$.  Finally, a direct
telescoping calculation with \eqref{eq:first-Yangian-ladder-modes}
gives
\begin{equation}
 [e_1,f_1]=2E=\psi_2,
 \qquad [e_1,f_0]=[e_0,f_1]=0=\psi_1,
\end{equation}
which checks the first instances of \eqref{eq:Yangian-ef}.  Formula
\eqref{eq:psi3-eigenvalue} follows from
\cref{thm:Delta2-explicit}.
\end{proof}

The next modes already show the form of the Yangian deformation.  We
use the conventions $p_0=D_0=0$ in the formulas below.

\begin{proposition}\label{prop:second-Yangian-modes}
The second raising and lowering modes are recovered from the Jack
cut-and-join operator by
\begin{equation}\label{eq:e2-f2-commutator}
  {
 e_2=\frac12[\Delta_2(\alpha),e_1],\qquad
 f_2=-\frac12[\Delta_2(\alpha),f_1].}
\end{equation}
Their normally ordered expressions are
\begin{align}
e_2={}&
 \frac{\alpha}{2}\sum_{i,j\geq1}(i+j)p_ip_jD_{i+j-1}
 -\alpha\sum_{i,j\geq1}i p_{i+1}p_jD_{i+j}
 \notag\\
&+\sum_{i,j\geq1}i p_{i+j}D_{i-1}D_j
 -\frac12\sum_{i,j\geq1}(i+j)p_{i+j+1}D_iD_j
 \notag\\
&+\frac{\alpha-1}{2}\sum_{k\geq1}k(k-1)
       \bigl(p_kD_{k-1}-p_{k+1}D_k\bigr),
\label{eq:e2-normal-ordered}\\[4pt]
f_2={}&
 -\alpha\sum_{i,j\geq1}i p_{i-1}p_jD_{i+j}
 +\frac{\alpha}{2}\sum_{i,j\geq1}(i+j)p_ip_jD_{i+j+1}
 \notag\\
&-\frac12\sum_{i,j\geq1}(i+j)p_{i+j-1}D_iD_j
 +\sum_{i,j\geq1}i p_{i+j}D_{i+1}D_j
 \notag\\
&-\frac{\alpha-1}{2}\sum_{k\geq1}k(k-1)
       \bigl(p_{k-1}D_k-p_kD_{k+1}\bigr).
\label{eq:f2-normal-ordered}
\end{align}
In particular, these operators satisfy the next relations
\begin{equation}\label{eq:second-Yangian-low-checks}
 [e_2,f_0]=[e_0,f_2]=\psi_2=2E.
\end{equation}
\end{proposition}

\begin{proof}
Since $[E,e_1]=e_1$, equations
\eqref{eq:Yangian-e-recursion}, \eqref{eq:psi3-Delta2-exact} and
$\sigma_3\psi_0=\alpha-1$ give
\begin{align*}
6e_2
 &= [\psi_3,e_1]-2\sigma_3\psi_0e_1\\
 &=3[\Delta_2(\alpha),e_1].
\end{align*}
The argument for $f_2$ is identical, using $[E,f_1]=-f_1$.

For completeness, set
\begin{equation}
 X=e_1=\sum_{r\geq1}p_{r+1}D_r,\qquad
 Y=-f_1=\sum_{r\geq1}p_rD_{r+1}.
\end{equation}
The elementary commutators
\begin{align}
 [X,p_k]&=k p_{k+1},& [X,D_k]&=-kD_{k-1},
 \label{eq:X-elementary-commutators}\\
 [Y,p_k]&=k p_{k-1},& [Y,D_k]&=-kD_{k+1}
 \label{eq:Y-elementary-commutators}
\end{align}
applied to the three lines of
\eqref{eq:Delta2-explicit-new} give
\eqref{eq:e2-normal-ordered} and
\eqref{eq:f2-normal-ordered}.  Finally, the Jacobi identity and the
relations already verified in \cref{prop:first-Yangian-modes} yield
\begin{align*}
 [e_2,f_0]
 &=\frac12[[\Delta_2,e_1],f_0]
   =[e_1,f_1]=\psi_2,\\
 [e_0,f_2]
 &=-\frac12[e_0,[\Delta_2,f_1]]
   =[e_1,f_1]=\psi_2.
\end{align*}
\end{proof}

With the parameters \eqref{eq:Yangian-parameters}, one has
$\sigma_3=0$ at $\alpha=1$.  The correction terms in
\eqref{eq:Yangian-psi3-e}--\eqref{eq:Yangian-psi3-f} then disappear,
and the affine-Yangian recursion specializes to the linear
$U(\mathcal W_{1+\infty})$ construction.  This, rather than the
Feigin--Fuchs completion, is the precise continuation of the
picture of \cite{FrenkelWang,LascouxThibon}.

%\begin{remark}[An auxiliary charge-one Virasoro]
%The previous version of this argument introduced the power-sum
%diagonal operator
%\begin{equation}
% H=\sum_{k\geq1}k p_kD_k
%\end{equation}
%and replaced $\Delta_2(\alpha)$ by
%$\Delta_2(\alpha)-(\alpha-1)H$.  The normalized commutators then become
%$L_n^{(0)}+\kappa a_n$, with
%$\kappa=-(\alpha-1)/(2\alpha)$; after shifting $L_0$ this gives a
%charge-one Virasoro algebra.  The calculation is correct, but $H$ is
%diagonal in the power-sum basis, not in the Jack basis.  The corrected
%operator is therefore not a Jack Hamiltonian, and this auxiliary
%Virasoro algebra should not be interpreted as the natural Jack
%deformation of the Frenkel--Wang construction.
%\end{remark}

\section{The content series in the deformed algebra}
\label{sec:deformed-W-conclusion}

The parallel with \cite{LascouxThibon} can now be stated at the level
of the first two nontrivial class operators.  The classical vertex
operator packages the commuting zero modes given by ordinary content
power sums.  In the Jack case, the corresponding degree-zero Cartan
series is
\begin{equation}\label{eq:deformed-Cartan-series-final}
 \mathbf C^{(\alpha)}(v)
 =\frac{\cQ^{\exp}(\alpha v)}
 {(1-e^{-v})(e^{\alpha v}-1)},
\end{equation}
and its eigenvalue on $P_\lambda^{(\alpha)}$ is
$\sum_{\square\in\lambda}e^{v c_\alpha(\square)}$.  Formula
\eqref{eq:deformed-Cartan-series-final} is therefore the exact Jack
analogue of the zero-mode part of the series \eqref{eq:LT-D-vertex} 
of \cite{LascouxThibon}.

In the affine-Yangian presentation, the coefficients of
\eqref{eq:deformed-Cartan-series-final} belong to the commuting Cartan
subalgebra.  The positive and negative halves are generated from the
first creation and annihilation operators by the defining commutators
of $\mathbf{SH}^c$.  At $\alpha=1$ these relations linearize to those
of $U(\mathcal W_{1+\infty})$, and the usual vertex
operator simultaneously displays all modes.  For generic $\alpha$,
the analogous object is not expected to be an ordinary
$\widehat{\mathfrak{gl}}_\infty$ vertex operator: its natural
relations are the nonlinear affine-Yangian relations.

The calculation of \cref{sec:Jack-Virasoro} identifies the first
nontrivial layer of this deformation explicitly.  The Heisenberg modes
remain undeformed after normalization, and the content Hamiltonian
$2\mathbf C_1^{(\alpha)}$ yields the Jack cut-and-join operator after
Hall duality.  Its first commutators display the Feigin--Fuchs
linear term, while the affine-Yangian relations specify the recursive
correction needed to generate the full raising and lowering currents.
This is the appropriate deformation of the first step of the 
construction of \cite{LascouxThibon}.

\section{Conclusion}
\label{sec:conclusion}

The classical passage from Goulden's cut-and-join operator to Virasoro,
and then from content power sums to $\mathcal W_{1+\infty}$, admits a
natural Jack deformation inside the stable spherical degenerate double
affine Hecke algebra $\mathbf{SH}^c$.  The essential computational
ingredient is the vacuum-normalized logarithm of the
Heckman--Polychronakos hierarchy.  Its eigenvalues are finite
differences of shifted row power sums, and Bernoulli polynomials convert
them uniformly into power sums of $\alpha$-contents.  The single series
\eqref{eq:content-generating-series} performs this conversion before
coefficient extraction.

Because our multiplication convention is dual to the usual spectral
one, the Jack cut-and-join operators are Hall adjoints of these
Hamiltonians.  This gives a direct proof of the formula for
$\Delta_2(\alpha)$.  At the next order, the normal-ordering recursion of
Sergeev and Veselov, together with the triangular correction between
their hierarchy and the Heckman--Polychronakos hierarchy, gives the
stable form of $\cQ_4$ and proves the explicit formula for
$\Delta_3(\alpha)$.  The cancellation of $p_0$ and of negative powers
of $\alpha$ suggests a general integral form for these operators.
The raw commutators with Heisenberg contain one half of the familiar
Feigin--Fuchs realization.  Their Virasoro completion is useful for
comparison with Selberg--Aomoto constraints, but the analogue of the
construction of \cite{FrenkelWang,LascouxThibon} is instead the recursive
completion by the raising, lowering and Cartan currents of the affine
Yangian.

The ambient-algebra interpretation also clarifies what should and
should not be expected at higher order.  At $\alpha=1$, the Fock module
is the charge-zero sector of the level-one
$\widehat{\mathfrak{gl}}_\infty$ representation and the operators lie
in the image of $U(\mathcal W_{1+\infty})$.  For generic $\alpha$, the
natural object is instead the Fock representation of $\mathbf{SH}^c$,
or equivalently the affine Yangian of $\mathfrak{gl}_1$ with parameters
proportional to $(\alpha,-1,1-\alpha)$.  Its defining relations give an
abstract recursive construction of all raising and lowering modes from
the first modes and the content Hamiltonian.  Determining their closed
normal-ordered expressions in our conventions remains a separate
problem in general; \cref{prop:first-Yangian-modes,prop:second-Yangian-modes}
carry it out explicitly through $e_2$ and $f_2$.  The calculation of
$\Delta_3(\alpha)$ shows, however, that the Dunkl operator at infinity
provides an effective algorithm at low order.  Establishing uniformly
that this algorithm produces an integral form over
$\mathbb Z[\alpha]$, and relating that form to positivity in
$b=\alpha-1$, are natural problems left open by the present paper.

\section*{Acknowledgments}
The author used ChatGPT (OpenAI) as an interactive aid in organizing
the manuscript, editing the English text, and checking
computer-algebra calculations.  All mathematical statements and
conclusions remain the responsibility of the author.

\enlargethispage{4\baselineskip}
\bibliographystyle{alpha}
\bibliography{jack_w_infinity}
\end{document}